\documentclass[10pt, a4paper, twocolumn ,titlepage, egregdoesnotlikesansseriftitles]{scrartcl}
\pdfoutput=1
\usepackage{geometry}
\usepackage{mathtools} 
\usepackage{stmaryrd} 
\usepackage{amsmath} 
\usepackage{amsthm} 
\usepackage{amssymb} 
\usepackage[english]{babel} 
\usepackage[hidelinks]{hyperref}
\usepackage[all]{xy} 
\usepackage{graphicx} 
\usepackage{chngpage} 
\usepackage{indentfirst} 
\usepackage[T1]{fontenc}
\usepackage[latin9]{inputenc}
\usepackage{prettyref} 
\usepackage{xcolor}
\usepackage[font={small,it}]{caption} 
\usepackage{sectsty} 

\setkomafont{section}{\large}
\graphicspath{{figures/}}

\numberwithin{equation}{section}
\numberwithin{figure}{section}

\theoremstyle{plain}
\newtheorem{thm}{Theorem}
\newtheorem{lem}[thm]{Lemma}
\newtheorem{prop}[thm]{Proposition}

\theoremstyle{definition}
\newtheorem*{defn*}{Definition}
\newtheorem*{example*}{Example}

\newcommand{\orcid}[1]{\href{https://orcid.org/#1}{\includegraphics[width=10pt]{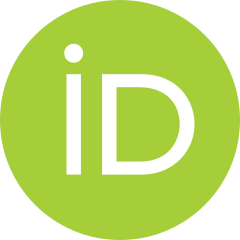}  orcid.org/#1}}

\title{\vspace{-0.5cm}\LARGE \textsc{\scalebox{0.92}[1.0]{A Generalization of the Pearson Correlation}\\\scalebox{0.92}[1.0]{to Riemannian Manifolds}}\vspace{-0.5cm}}
\author{\large P. Michl \orcid{0000-0002-6398-0654}}
\date{}

\begin{document}
\twocolumn[ 
\begin{@twocolumnfalse}  

\maketitle

\begin{abstract}
\vspace{-1.9cm}
\begin{adjustwidth}{10mm}{10mm}
The increasing application of deep-learning is accompanied by a shift towards highly non-linear statistical models. In terms of their geometry it is natural to identify these models with Riemannian manifolds. The further analysis of the statistical models therefore raises the issue of a correlation measure, that in the cutting planes of the tangent spaces equals the respective Pearson correlation and extends to a correlation measure that is normalized with respect to the underlying manifold. In this purpose the article reconstitutes elementary properties of the Pearson correlation to successively derive a linear generalization to multiple dimensions and thereupon a nonlinear generalization to principal manifolds, given by the Riemann-Pearson Correlation.
\end{adjustwidth}
\vspace{0.5cm}
\end{abstract}

\end{@twocolumnfalse}
] 

\section{Introduction}

A fundamental issue, that accompanies the analysis of multivariate
data, concerns the quantification of statistically dependency structures
by association measures. Many approaches in this direction can be
traced back to the late 19\textsuperscript{th} century, where the
issue was closely related to the task, to extract laws of nature from
two dimensional scatter plots. This in particular applies to the widespread
Pearson correlation coefficient.
\begin{defn*}[Pearson Correlation]
\label{def:Pearson Correlation}\emph{Let $X\colon\Omega\to\mathbb{R}$
and $Y\colon\Omega\to\mathbb{R}$ be random variables with finite
variances $\sigma_{X}^{2}$ and $\sigma_{Y}^{2}$. Then the Pearson
correlation $\rho_{X,Y}$ is defined by:
\begin{equation}
\rho_{X,Y}\coloneqq\frac{\mathrm{Cov}(X,\,Y)}{\sigma_{X}\sigma_{Y}}\label{eq:Pearson Correlation}
\end{equation}
}
\end{defn*}
Due to its popularity and simplicity the Pearson correlation has been
generalized to a variety of different domains of application, including
generic monotonous relationships, relationships between sets of random
variables and asymmetric relationships (\nocite{Zheng2010}Zheng et
al. 2010). In the purpose to provide a generalization to smooth curves
and submanifolds, that allow an incorporation of structural assumptions,
some elementary considerations have to be taken into account, that
allow a separation between the pairwise quantification of dependencies
and their global modelling.

\section{Correlation and Regression Dilution}

Pearson's original motivation, was the regression of
a straight line, that minimizes the averaged Euclidean distance to
points, that are scattered about it (\nocite{Pearson1901}Pearson
1901, p561). Thereby his investigations were preceded by the observation,
that for a measurement series the assumed ``direction of causality''
influences the estimate of the slope of the regression line. Thereby
the direction of causality is implicated by the choice of an error
model, that assumes one random variable to be error free and the other
to account for the whole observed error. Pearson empirically observed,
that for $n\in\mathbb{N}$ points, given by i.i.d. realizations $\boldsymbol{x}\in\mathbb{R}^{n}$
of $X$ and $\boldsymbol{y}\in\mathbb{R}^{n}$ of $Y$, the\textbf{
least squares} regression line of $\boldsymbol{y}$ on $\boldsymbol{x}$
only equals the regression line of $\boldsymbol{x}$ on $\boldsymbol{y}$,
if all points perfectly fit on a straight line. In all other cases,
however, the slopes of the respective regression lines turned out,
not to be reciprocal and their product was found within the interval
$[0,\,1)$. This observation was decisive for Pearson's
definition of the correlation coefficient. Thereby $\rho_{X,Y}$ is
estimated by its empirical counterpart $\rho_{x,y}$, that replaces
variances by sample variances and the covariance by the sample variance.
\begin{lem}
\label{lem:Correlation Lin. Regr. Slopes}Let $X\colon\Omega\to\mathbb{R}$
and $Y\colon\Omega\to\mathbb{R}$ be random variables with $n\in\mathbb{N}$
i.i.d. realizations $\boldsymbol{x}\in\mathbb{R}^{n}$ and $\boldsymbol{y}\in\mathbb{R}^{n}$.
Furthermore let $\beta_{x}\in\mathbb{R}$ denote the slope of the
linear regression of $\boldsymbol{y}$ on $\boldsymbol{x}$ and $\beta_{y}\in\mathbb{R}$
the slope of the linear regression of $\boldsymbol{x}$ on $\boldsymbol{y}$.
Then:
\begin{equation}
\rho_{x,y}^{2}=\beta_{x}\beta_{y}\label{eq:Correlation Lin. Regr. Slopes}
\end{equation}
\end{lem}

\begin{proof}[Proof of Lemma \ref{lem:Correlation Lin. Regr. Slopes}]
The following proof is based on (\nocite{Kenney1962}Kenney et al.
1962). The least squares regression of $\boldsymbol{y}$ on $\boldsymbol{x}$
implicates, that for regression coefficients $\alpha_{x},\beta_{x}\in\mathbb{R}$
and a normal distributed random error$\varepsilon\coloneqq Y-(\beta_{x}X+\alpha_{x})$
the log-likelihood of the realizations is maximized, if and only if
the $\ell^{2}$-norm of the realizations of $\varepsilon$ is minimized,
such that:
\begin{equation}
{\displaystyle SSE_{y}(\alpha_{x},\,\beta_{x})\coloneqq\sum_{i=1}^{n}(y_{i}-(\beta_{x}x_{i}+\alpha_{x}))^{2}}\to\min\label{eq:Correlation Lin. Regr. Slopes 1}
\end{equation}
Since $SSE_{y}$ is a quadratic function of $\alpha_{x}$ and $\beta_{x}$
and therefore convex, it has a unique global minimum at:
\begin{align}
\frac{\partial}{\partial\alpha_{x}}SSE_{y} & =2\sum_{i=1}^{n}(y_{i}-(\beta_{x}x_{i}+\alpha_{x}))(-x_{i})=0\label{eq:Correlation Lin. Regr. Slopes 2}\\
\frac{\partial}{\partial\beta_{x}}SSE_{y} & =2\sum_{i=1}^{n}(y_{i}-(\beta_{x}x_{i}+\alpha_{x}))(-1)=0\label{eq:Correlation Lin. Regr. Slopes 3}
\end{align}
By equating the coefficients, equations \ref{eq:Correlation Lin. Regr. Slopes 2}
and \ref{eq:Correlation Lin. Regr. Slopes 3} can be rewritten as
a system of linear equations of $\alpha_{x}$ and $\beta_{x}$:
\begin{align}
\alpha_{x}n+\beta_{x}\sum_{i=1}^{n}x_{i} & =\sum_{i=1}^{n}y_{i}\label{eq:Correlation Lin. Regr. Slopes 4}\\
\alpha_{x}\sum_{i=1}^{n}x_{i}+\beta_{x}\sum_{i=1}^{n}x_{i}^{2} & =\sum_{i=1}^{n}x_{i}y_{i}\label{eq:Correlation Lin. Regr. Slopes 5}
\end{align}
Consequently in matrix notation the vector $(\alpha_{x},\,\beta_{x})^{\mathrm{T}}$
is determined by:
\begin{equation}
\left(\begin{array}{c}
\alpha_{x}\\
\beta_{x}
\end{array}\right)=\left(\begin{array}{cc}
n & \sum_{i=1}^{n}x_{i}\\
\sum_{i=1}^{n}x_{i} & \sum_{i=1}^{n}x_{i}^{2}
\end{array}\right)^{-1}\left(\begin{array}{c}
\sum_{i=1}^{n}y_{i}\\
\sum_{i=1}^{n}x_{i}y_{i}
\end{array}\right)\label{eq:Correlation Lin. Regr. Slopes 6}
\end{equation}
Let $\overline{x},\,\overline{y}$ respectively denote the \emph{sample
means}. Then by calculating the matrix inverse, the slope $\beta_{x}$
equates to:
\begin{equation}
\beta_{x}=\left(\sum_{i=1}^{n}x_{i}y_{i}-n\overline{x}\overline{y}\right)\left(\sum_{i=1}^{n}x_{i}^{2}-n\overline{x}^{2}\right)^{-1}\label{eq:Correlation Lin. Regr. Slopes 7}
\end{equation}
Thereupon by substituting the \emph{sample variance}:
\begin{align}
\sigma_{x}^{2} & \coloneqq\frac{1}{n}\sum_{i=1}^{n}(x_{i}-\overline{x})^{2}\nonumber \\
 & =\frac{1}{n}\left(\sum_{i=1}^{n}x_{i}^{2}-n\overline{x}^{2}\right)\label{eq:Correlation Lin. Regr. Slopes 9}
\end{align}
And the \emph{sample covariance}:
\begin{align}
\mathrm{Cov}(\boldsymbol{x},\,\boldsymbol{y}) & \coloneqq\frac{1}{n}\sum_{i=1}^{n}(x_{i}-\overline{x})(y_{i}-\overline{y})\nonumber \\
 & =\frac{1}{n}\left(\sum_{i=1}^{n}x_{i}y_{i}-n\overline{x}\overline{y}\right)\label{eq:Correlation Lin. Regr. Slopes 8}
\end{align}
It follows from equation \ref{eq:Correlation Lin. Regr. Slopes 7},
that:
\begin{equation}
\beta_{x}=\frac{\mathrm{Cov}(\boldsymbol{x},\,\boldsymbol{y})}{\sigma_{x}^{2}}\label{eq:Correlation Lin. Regr. Slopes 10}
\end{equation}
Conversely the slope $\beta_{y}$ of the linear regression of $\boldsymbol{x}$
on $\boldsymbol{y}$ mutatis mutandis equates to: 
\begin{equation}
\beta_{y}=\frac{\mathrm{Cov}(\boldsymbol{y},\,\boldsymbol{x})}{\sigma_{y}^{2}}\label{eq:Correlation Lin. Regr. Slopes 11}
\end{equation}
By the symmetry of $\mathrm{Cov}$ it the follows, from equations
\ref{eq:Correlation Lin. Regr. Slopes 10} and \ref{eq:Correlation Lin. Regr. Slopes 11}
that:
\begin{equation}
\beta_{x}\beta_{y}=\frac{\mathrm{Cov}(\boldsymbol{x},\,\boldsymbol{y})^{2}}{\sigma_{x}^{2}\sigma_{y}^{2}}=\rho_{xy}^{2}\label{eq:Correlation Lin. Regr. Slopes 12}
\end{equation}
\end{proof}
Lemma \ref{lem:Correlation Lin. Regr. Slopes} shows, that $\rho_{x,y}$
may be regarded as the geometric mean of the regression slopes $\beta_{x}$
and $\beta_{y}$, where $\beta_{x}$ and $\beta_{y}$ respectively
describe the causal relationships $X\to Y$ and $Y\to X$. Thereby
$X$ and $Y$ respectively are treated as error free regressor variables
to predict the corresponding response variable, that captures the
overall error. The mutual linear relationship $X\leftrightarrow Y$
is then described by a regression line, that equally treats errors
in both variables. As an immediate consequence of this symmetry it
follows, that this \textbf{total least squares} regression line is
unique, and its slope $\beta_{x}^{\star}$, that describes $\boldsymbol{y}$
by $\boldsymbol{x}$ is reciprocal to the slope $\beta_{y}^{\star}$,
that describes $\boldsymbol{x}$ by $\boldsymbol{y}$ such that $\beta_{x}^{\star}\beta_{y}^{\star}=1$.
In this sense $\beta_{x}$ and $\beta_{y}$ may be regarded as biased
estimations of $\beta_{x}^{\star}$ and $\beta_{y}^{\star}$. Thereby
the bias generally is known as ``regression dilution'' or ``regression
attenuation''. For the case that both errors are independent and
normal distributed, this bias can be corrected by a prefactor, that
incorporates the error of the respective regressor variable. An application
of this correction to lemma \ref{lem:Correlation Lin. Regr. Slopes}
then shows, that $\rho_{X,Y}$ has a consistent estimations by the
sample variances of $\boldsymbol{x}$ and $\boldsymbol{y}$ and the
variances of their respective errors $\varepsilon_{X}$ and $\varepsilon_{Y}$.
\begin{prop}
\label{prop:Correlation Reliability}Let $X\colon\Omega\to\mathbb{R}$
and $Y\colon\Omega\to\mathbb{R}$ be random variables with $n\in\mathbb{N}$
i.i.d. realizations $\boldsymbol{x}\in\mathbb{R}^{n}$ and $\boldsymbol{y}\in\mathbb{R}^{n}$
and random errors $\varepsilon_{X}\sim\mathcal{N}(0,\,\eta_{X}^{2})$
and $\varepsilon_{Y}\sim\mathcal{N}(0,\,\eta_{Y}^{2})$. Then:
\begin{equation}
\rho_{x,y}^{2}\overset{{\scriptstyle P}}{\to}\left(1-\frac{\eta_{X}^{2}}{\sigma_{x}^{2}}\right)\left(1-\frac{\eta_{Y}^{2}}{\sigma_{y}^{2}}\right),\:\text{for }n\to\infty\label{eq:Correlation Reliability}
\end{equation}
\end{prop}

\begin{proof}[Proof of Proposition \ref{prop:Correlation Reliability}]
Let $\beta_{x}\in\mathbb{R}$ be the slope of the ordinary least
squares (OLS) regression line of $\boldsymbol{y}$ on $\boldsymbol{x}$,
where $\boldsymbol{x}$ is assumed to realize $X$ with a normal distributed
random error $\varepsilon_{X}\sim\mathcal{N}(0,\,\eta_{X}^{2})$.
Then $X$ decomposes into (i) an unobserved error free regressor variable
$X^{\star}$ and (ii) the random error $\varepsilon_{X}$, such that:
\begin{equation}
X\sim X^{\star}+\varepsilon_{X}\label{eq:Correlation Reliability, proof 1}
\end{equation}
With respect to this decomposition, the slope $\beta_{x}^{\star}$
of the total least squares (TLS) regression line, that also considers
$\varepsilon_{X}$, then is identified by the slope of the OLS regression
of $\boldsymbol{y}$ on $\boldsymbol{x}^{\star}$, where $\boldsymbol{x}^{\star}$
realizes $X^{\star}$. Thereupon let $\sigma_{x^{*}}^{2}$ be the
empirical variance of $\boldsymbol{x}^{\star}$, then according to
(\nocite{Snedecor1967}Snedecor et al. 1967) it follows, that:
\begin{equation}
\beta_{x}\overset{{\scriptstyle P}}{\to}\frac{\sigma_{x^{*}}^{2}}{\sigma_{x^{*}}^{2}+\eta_{X}^{2}}\beta_{x}^{\star},\:\text{for }n\to\infty\label{eq:Correlation Reliability, proof 2}
\end{equation}
Since furthermore $\varepsilon_{X}$ by definition is statistically
independent from $X^{\star}$, it can be concluded, that:
\begin{align}
\sigma_{x}^{2} & =\mathrm{Var}(X^{\star}+\varepsilon_{X})\nonumber \\
 & =\mathrm{Var}(X^{\star})+\mathrm{Var}(\varepsilon_{X})=\sigma_{x^{*}}^{2}+\eta_{X}^{2}\label{eq:Correlation Reliability, proof 3}
\end{align}
Such that:
\begin{equation}
\frac{\sigma_{x^{*}}^{2}}{\sigma_{x^{*}}^{2}+\eta_{X}^{2}}\stackrel{{\scriptstyle \text{\ref{eq:Correlation Reliability, proof 3}}}}{=}\frac{\sigma_{x}^{2}-\eta_{X}^{2}}{\sigma_{x}^{2}}=1-\frac{\eta_{X}^{2}}{\sigma_{x}^{2}}\label{eq:Correlation Reliability, proof 4}
\end{equation}
And therefore by equation \ref{eq:Correlation Reliability, proof 2}
that:
\begin{equation}
\beta_{x}\overset{{\scriptstyle P}}{\to}\left(1-\frac{\eta_{X}^{2}}{\sigma_{x}^{2}}\right)\beta_{x}^{\star},\:\text{for }n\to\infty\label{eq:Correlation Reliability, proof 5}
\end{equation}
Conversely let now $\beta_{y}\in\mathbb{R}$ be the the slope of the
OLS regression of $\boldsymbol{x}$ on $\boldsymbol{y}$, where $\boldsymbol{y}$
is assumed to realize $Y$ with a random error $\varepsilon_{X}\sim\mathcal{N}(0,\,\eta_{Y}^{2})$.
Then also the corrected slope $\beta_{y}^{\star}$ mutatis mutandis
satisfies the relation given by equation \ref{eq:Correlation Reliability, proof 5}
and by the representation of $\rho_{x,y}$, as given by lemma \ref{lem:Correlation Lin. Regr. Slopes},
it then can be concluded, that:

\begin{align}
\rho_{x,y}^{2} & \stackrel{{\scriptstyle \text{\ref{lem:Correlation Lin. Regr. Slopes}}}}{=}\beta_{x}\beta_{x}\label{eq:Correlation Reliability, proof 6}\\
 & \overset{{\scriptstyle P}}{\to}\left(1-\frac{\eta_{X}^{2}}{\sigma_{x}^{2}}\right)\left(1-\frac{\eta_{Y}^{2}}{\sigma_{y}^{2}}\right)\beta_{x}^{\star}\beta_{y}^{\star},\:\text{for }n\to\infty\nonumber 
\end{align}
The proposition then follows by the uniqueness of the total least
squares regression line for known variances $\eta_{X}^{2}$ and $\eta_{Y}^{2}$,
such that:
\[
\beta_{y}^{\star}=\frac{1}{\beta_{x}^{\star}}
\]
\end{proof}

\section{A Generalization to linear Principal Manifolds}

Within the same publication, in which Pearson introduced the correlation
coefficient, he also developed a structured approach that determines
the straight line, that minimizes the Euclidean distance (\nocite{Pearson1901}Pearson
1901, p563). His method, which later received attribution as the method
of \textbf{Principal Component Analysis} (\textbf{PCA}), however,
even went further and allowed a canonical generalization of the problem
in the following sense: For $d\in\mathbb{N}$ let $\boldsymbol{X}\colon\Omega\to\mathbb{R}^{d}$
be a multivariate random vector and for $n\in\mathbb{N}$ let $\boldsymbol{x}\in\mathbb{R}^{n\times d}$
be an i.i.d. realization of $\boldsymbol{X}$. Then for any given
$k\in\mathbb{N}$ with $k\leq d$ the goal is, to determine an affine
linear subspace $L\subseteq\mathbb{R}^{d}$ of dimension $k$, that
minimizes the summed Euclidean distance to $\boldsymbol{x}$. In order
to solve this problem, the fundamental idea of Pearson was, to transfer
the principal axis theorem from ellipsoids to multivariate Gaussian
distributed random vectors. Thereupon, however, the method also can
be formulated with respect to generic elliptical distributions.
\begin{defn*}[Elliptical Distribution]
\label{def:Elliptical Distribution}\emph{For $d\in\mathbb{N}$ let
$\boldsymbol{X}\colon\Omega\to\mathbb{R}^{d}$ be a random vector.
Then $\boldsymbol{X}$ is elliptically distributed, iff there exists
a random vector $\boldsymbol{S}\colon\Omega\to\mathbb{R}^{k}$ with
$k\leq d$, which distribution is invariant to rotations, a matrix
$A\in\mathbb{R}^{d\times k}$ of rank $k$ and a vector $\boldsymbol{b}\in\mathbb{R}^{d}$,
such that:}
\begin{equation}
\boldsymbol{X}\sim A\boldsymbol{S}+\boldsymbol{b}\label{eq:Elliptical Distribution}
\end{equation}
\end{defn*}
Consequently a random vector $\boldsymbol{X}$ is elliptically distributed,
if it can be represented by an affine transformation of a radial symmetric
distributed random vector $\boldsymbol{S}$. The decisive property,
that underpins the choice of elliptical distributions, lies within
their coincidence of linear and statistical dependencies, which allows
to decompose $\boldsymbol{X}$ in statistically independent components
by a linear decomposition. This property allows, to substantiate the
multidimensional ``linear fitting problem'' with respect to an orthogonal
projection.
\begin{prop}
\label{prop:Affine fitting of Elliptic distributions}For $d\in\mathbb{N}$
let $\boldsymbol{X}\colon\Omega\to\mathbb{R}^{d}$ be an elliptically
distributed random vector, $L\subseteq\mathbb{R}^{d}$ an affine linear
subspace of $\mathbb{R}^{d}$ and $\pi_{L}$ the orthogonal projection
of $\mathbb{R}^{d}$ onto $L$. Then the following statements are
equivalent:

\emph{(i)} $L$ minimizes the Euclidean distance to $\boldsymbol{X}$

\emph{(ii)} $\mathbb{E}(\boldsymbol{X})\in L$ and $L$ maximizes
the variance $\mathrm{Var}(\pi_{L}(\boldsymbol{X}))$ 
\end{prop}

\begin{proof}
Let $\boldsymbol{Y}_{L}\coloneqq\boldsymbol{X}-\pi_{L}(\boldsymbol{X})$,
then the Euclidean distance between $\boldsymbol{X}$ and $L$ can
be written:

\begin{align}
d(\boldsymbol{X},\,\pi_{L}(\boldsymbol{X}))^{2} & =\mathbb{E}(\|(\boldsymbol{X}-\pi_{L}(\boldsymbol{X}))\|_{2}^{2})\label{eq:L1.1.1}\\
 & =\mathbb{E}(\boldsymbol{Y}_{L}^{2})\nonumber 
\end{align}
This representation can furthermore be decomposed by using the algebraic
formula for the variance:
\begin{equation}
\mathbb{E}(\boldsymbol{Y}_{L}^{2})=\mathrm{Var}(\boldsymbol{Y}_{L})+\mathbb{E}(\boldsymbol{Y}_{L})^{2}\label{eq:L1.1.2}
\end{equation}
Let now be $\boldsymbol{Y}_{L}^{\perp}\coloneqq\pi_{L}(\boldsymbol{X})$,
then $\boldsymbol{X}=\boldsymbol{Y}_{L}+\boldsymbol{Y}_{L}^{\perp}$
and $\boldsymbol{Y}_{L}$ and $\boldsymbol{Y}_{L}^{\perp}$ are uncorrelated,
such that: 
\begin{align}
\mathrm{Var}(\boldsymbol{X}) & =\mathrm{Var}(\boldsymbol{Y}_{L}+\boldsymbol{Y}_{L}^{\perp})\label{eq:L1.1.3}\\
 & =\mathrm{Var}(\boldsymbol{Y}_{L})+\mathrm{Var}(\boldsymbol{Y}_{L}^{\perp})\nonumber 
\end{align}
From equations \ref{eq:L1.1.1}, \ref{eq:L1.1.2} and \ref{eq:L1.1.3}
it follows, that:

\begin{equation}
d(\boldsymbol{X},\,\pi_{L}(\boldsymbol{X}))^{2}=\mathrm{Var}(\boldsymbol{X})-\mathrm{Var}(\boldsymbol{Y}_{L}^{\perp})+\mathbb{E}(\boldsymbol{Y}_{L})^{2}\label{eq:L1.1.4}
\end{equation}
Consequentially the Euclidean distance is minimized, if and only if
the right side of equation \ref{eq:L1.1.4} is minimized. The first
term $\mathrm{Var}(\boldsymbol{X})$, however, does not depend on
$L$ and since $\boldsymbol{X}$ is elliptically distributed, the
linear independence of $\boldsymbol{Y}_{L}^{\perp}$ and $\boldsymbol{Y}_{L}$
is sufficient for statistically independence. It follows, that the
Euclidean distance is minimized, if and only if: (1) The term $\mathbb{E}(\boldsymbol{Y}_{L})^{2}$
is minimized and (2) the term $\mathrm{Var}(\boldsymbol{Y}_{L}^{\perp})$
is maximized. Concerning (1) it follows, that:

\begin{align*}
\mathbb{E}(\boldsymbol{Y}_{L})^{2} & =\mathbb{E}(\boldsymbol{X}-\pi_{L}(\boldsymbol{X}))^{2}\\
 & =(\mathbb{E}(\boldsymbol{X})-\pi_{L}(\mathbb{E}(\boldsymbol{X})))^{2}
\end{align*}
Therefore the term $\mathbb{E}(\boldsymbol{Y}_{L})^{2}$ is minimized,
if and only if $\pi_{L}(\mathbb{E}(\boldsymbol{X}))=\mathbb{E}(\boldsymbol{X})$,
which in turn means that $\mathbb{E}(\boldsymbol{X})\in L$. Concerning
(2), the proposition immediately follows by the definition of $\boldsymbol{Y}_{L}^{\perp}$.
\end{proof}
Let now be $k\leq d$. In order to derive an affine linear subspace
$L\subseteq\mathbb{R}^{d}$ that minimizes the Euclidean distance
to $\boldsymbol{X}$, proposition ... states, that is suffices to
provide an $L$ which (1) is centred in $\boldsymbol{X}$, such that
$\mathbb{E}(\boldsymbol{X})\in L$, and (2) maximizes the variance
of the projection. In order to maximize $\mathrm{Var}(\pi_{L}(\boldsymbol{X}))$,
however, it is beneficial to give a further representation.
\begin{lem}
\label{lem:Variance of Projection}For $d\in\mathbb{N}$ let $\boldsymbol{X}\colon\Omega\to\mathbb{R}^{d}$
be an elliptically distributed random vector, $L\subseteq\mathbb{R}^{d}$
an affine linear subspace of $\mathbb{R}^{d}$, which for an $k\leq d$,
a vector $\boldsymbol{v}\in\mathbb{R}^{d}$ and an orthonormal basis
$\boldsymbol{u}_{1},\,\ldots\,,\,\boldsymbol{u}{}_{k}\in\mathbb{R}^{d}$
is given by:
\[
L=\boldsymbol{v}+\bigoplus_{i=1}^{k}\mathbb{R}\boldsymbol{u}_{i}
\]
Let further be $\pi_{L}\colon\mathbb{R}^{d}\to L$ the orthogonal
projection of $\mathbb{R}^{d}$ onto $L$. Then the variance of the
projection is given by:
\[
\mathrm{Var}(\pi_{L}(\boldsymbol{X}))=\sum_{i=1}^{k}\boldsymbol{u}_{i}^{T}\mathrm{Cov}(\boldsymbol{X})\boldsymbol{u}_{i}
\]
\end{lem}

\begin{proof}
Let $L^{'}\coloneqq L+\mathbb{E}(\boldsymbol{X})-\boldsymbol{v}$,
then the orthogonal projection $\pi_{L^{'}}(\boldsymbol{X})$ decomposes
into individual orthogonal projections to the respective basis vectors,
such that:
\begin{equation}
\pi_{L^{'}}(\boldsymbol{X})=\mathbb{E}(\boldsymbol{X})+\sum_{i=1}^{k}\left\langle \boldsymbol{X}-\mathbb{E}(\boldsymbol{X}),\,\boldsymbol{u}_{i}\right\rangle \boldsymbol{u}_{i}\label{eq:L1.2.1}
\end{equation}
Let $\hat{X}_{i}\coloneqq\left\langle \boldsymbol{X},\,\boldsymbol{u}_{i}\right\rangle \boldsymbol{u}_{i}$,
for $i\in\{1,\,\ldots\,,\,k\}$. The total variance of this projection
is then given by:
\begin{align}
\mathrm{Var}(\pi_{L^{'}}(\boldsymbol{X})) & \stackrel{_{\ref{eq:L1.2.1}}}{=}\mathrm{Var}\left(\mathbb{E}(\boldsymbol{X})+\sum_{i=1}^{k}\hat{X}_{i}-\sum_{i=1}^{k}\mathbb{E}(\hat{X}_{i})\right)\label{eq:L1.2.2}\\
 & =\mathrm{Var}\left(\sum_{i=1}^{k}\hat{X}_{i}\right)\nonumber 
\end{align}
Since the random variables $\hat{X}_{i}$ by definition are uncorrelated,
the algebraic formula for the variance can be used to decompose the
variance:
\begin{align}
\mathrm{Var}\left(\sum_{i=1}^{k}\hat{X}_{i}\right) & =\sum_{i=1}^{k}\mathrm{Var}(\hat{X}_{i})\label{eq:L1.2.3}
\end{align}
By equating the term $\mathrm{Var}(\hat{X}_{i})$, for $i\in\{1,\,\ldots\,,\,k\}$
it follows, that:

\begin{align}
\mathrm{Var}(\hat{X}_{i}) & =\mathrm{Var}\left(\left\langle \boldsymbol{X},\,\boldsymbol{u}_{i}\right\rangle \boldsymbol{u}_{i}\right)\label{eq:L1.2.4}\\
 & =\mathrm{Var}\left(\boldsymbol{X}^{\mathrm{T}}\boldsymbol{u}_{i}\right)\boldsymbol{u}_{i}^{2}\nonumber \\
 & =\mathrm{Var}\left(\boldsymbol{X}^{\mathrm{T}}\boldsymbol{u}_{i}\right)\nonumber 
\end{align}
And furthermore by introducing the covariance matrix $\mathrm{Cov}(\boldsymbol{X})$:
\begin{align}
\mathrm{Var}\left(\boldsymbol{X}^{\mathrm{T}}\boldsymbol{u}_{i}\right) & \stackrel{_{\text{def}}}{=}\mathbb{E}\left((\boldsymbol{X}^{\mathrm{T}}\boldsymbol{u}_{i})^{\mathrm{T}}(\boldsymbol{X}^{\mathrm{T}}\boldsymbol{u}_{i})\right)\label{eq:L1.2.5}\\
 & =\boldsymbol{u}_{i}^{\mathrm{T}}\mathbb{E}\left(\boldsymbol{X}^{\mathrm{T}}\boldsymbol{X}\right)\boldsymbol{u}_{i}\nonumber \\
 & \stackrel{_{\text{def}}}{=}\boldsymbol{u}_{i}^{T}\mathrm{Cov}(\boldsymbol{X})\boldsymbol{u}_{i}\nonumber 
\end{align}
Summarized the equations \ref{eq:L1.2.2}, \ref{eq:L1.2.3}, \ref{eq:L1.2.4}
and \ref{eq:L1.2.5} provide a representation for the variance of
the projection to $L^{'}$:
\[
\mathrm{Var}(\pi_{L^{'}}(\boldsymbol{X}))=\sum_{i=1}^{k}\boldsymbol{u}_{i}^{T}\mathrm{Cov}(\boldsymbol{X})\boldsymbol{u}_{i}
\]
Finally the total variance of the projection is invariant under translations
of $L$, such that:
\begin{align}
\mathrm{Var}(\pi_{L}(\boldsymbol{X})) & =\mathrm{Var}(\pi_{L}(\boldsymbol{X})+\mathbb{E}(\boldsymbol{X})-\boldsymbol{v})\label{eq:Invariance of Variance to Parrallel Translation}\\
 & =\mathrm{Var}(\pi_{L^{'}}(\boldsymbol{X}))\nonumber \\
 & \stackrel{\ref{eq:L1.2.3}}{=}\sum_{i=1}^{k}\boldsymbol{u}_{i}^{T}\mathrm{Cov}(\boldsymbol{X})\boldsymbol{u}_{i}\nonumber 
\end{align}
\end{proof}
Lemma \ref{lem:Variance of Projection}, shows, that for elliptically
distributed random vectors $\mathbf{X}$ the best fitting linear subspaces
are completely determined by the expectation $\mathbb{E}(\boldsymbol{X})$
and the covariance matrix $\mathrm{Cov}(\boldsymbol{X})$. On this
point it is important to notice, that the covariance matrix is symmetric,
which allows its diagonalization with regard to real valued Eigenvalues.
\begin{lem}
\label{lem:Variance of Projection 2}For $d\in\mathbb{N}$ let $\boldsymbol{X}\colon\Omega\to\mathbb{R}^{d}$
be an elliptically distributed random vector, $L\subseteq\mathbb{R}^{d}$
an affine linear subspace of $\mathbb{R}^{d}$, which for an $k\leq d$,
a vector $\boldsymbol{v}\in\mathbb{R}^{d}$ and an orthonormal basis
$\boldsymbol{u}_{1},\,\ldots\,,\,\boldsymbol{u}{}_{k}\in\mathbb{R}^{d}$
is given by:
\[
L=\boldsymbol{v}+\bigoplus_{i=1}^{k}\mathbb{R}\boldsymbol{u}_{i}
\]
Let further be $\pi_{L}\colon\mathbb{R}^{d}\to L$ the orthogonal
projection of $\mathbb{R}^{d}$ onto $L$, as well as $\lambda_{1},\,\ldots\,,\,\lambda_{d}\in\mathbb{R}$
the eigenvalues of $\mathrm{Cov}(\boldsymbol{X})$. Then there exist
numbers $a_{1},\,\ldots\,,\,a_{d}\in[0,\,1]$ with $\sum_{i=1}^{d}a_{i}=k$,
such that:
\[
\mathrm{Var}(\pi_{L}(\boldsymbol{X}))=\sum_{i=1}^{d}\lambda_{i}a_{i}
\]
\end{lem}

\begin{proof}
From lemma \ref{lem:Variance of Projection} it follows, that:
\[
\mathrm{Var}(\pi_{L}(\boldsymbol{X}))=\sum_{j=1}^{k}\boldsymbol{u}_{j}^{T}\mathrm{Cov}(\boldsymbol{X})\boldsymbol{u}_{j}
\]
Since the covariance matrix $\mathrm{Cov}(\boldsymbol{X})$ is a symmetric
matrix, there exists an orthonormal basis transformation matrix $S\in\mathbb{R}^{d\times d}$
and a diagonal matrix $D\in\mathbb{R}^{d\times d}$, such that $\mathrm{Cov}(\boldsymbol{X})=S^{T}DS$.
Then the variance $\mathrm{Var}(\pi_{L}(\boldsymbol{X}))$ has a decomposition,
given by:

\begin{align*}
\mathrm{Var}(\pi_{L}(\boldsymbol{X})) & =\sum_{i=1}^{k}\boldsymbol{u}_{j}{}^{\mathrm{T}}S^{\mathrm{T}}DS\boldsymbol{u}_{j}\\
 & =\sum_{i=1}^{k}(S\boldsymbol{u}_{j})^{\mathrm{T}}DS\boldsymbol{u}_{j}
\end{align*}
For $j\in\{1,\,\ldots\,,\,k\}$ let now $\boldsymbol{c}_{j}\coloneqq S\boldsymbol{u}_{j}$
and for $i\in\{1,\,\ldots\,,\,n\}$ let the number $a_{i}\in\mathbb{R}$
be defined by: 
\[
a_{i}\coloneqq\sum_{j=1}^{k}(\boldsymbol{c}_{j}{}_{i})^{2}
\]
Then according to Lemma \ref{lem:Variance of Projection} the variance
$\mathrm{Var}(\pi_{L}(\boldsymbol{X}))$ can be decomposed:
\begin{align*}
\sum_{i=1}^{k}\boldsymbol{c}_{j}{}^{T}D\boldsymbol{c}_{j} & =\sum_{j=1}^{k}\sum_{i=1}^{d}\boldsymbol{c}_{j}{}_{i}\lambda_{i}\boldsymbol{c}_{j}{}_{i}\\
 & =\sum_{i=1}^{d}\lambda_{i}\sum_{j=1}^{k}(\boldsymbol{c}_{j}{}_{i})^{2}\\
 & =\sum_{i=1}^{d}\lambda_{i}a_{i}
\end{align*}
Furthermore since $\boldsymbol{u}_{1},\,\ldots\,,\,\boldsymbol{u}{}_{k}$
is an orthonormal basis and $S$ an orthonormal matrix it follows
that also $\boldsymbol{c}_{1},\,\ldots\,,\,\boldsymbol{c}{}_{k}$
is an orthonormal basis. Consequentially for $i\in\{1,\,\ldots\,,\,d\}$
it holds, that: 

\[
a_{i}=\sum_{j=1}^{k}(\boldsymbol{c}_{j}{}_{i})^{2}\leq\sum_{j=1}^{k}\left\Vert \boldsymbol{c}_{j}{}_{i}\right\Vert _{2}\leq1
\]
And furthermore by its definition it follows, that $a_{i}\geq0$,
such that $a_{i}\in[0,\,1]$. Besides this the sum over all $a_{i}$
equates to:

\begin{align*}
\sum_{i=1}^{d}a_{i} & =\sum_{i=1}^{d}\sum_{j=1}^{k}(\boldsymbol{c}_{j}{}_{i})^{2}\\
 & =\sum_{j=1}^{k}\boldsymbol{c}_{j}^{T}\boldsymbol{c}_{j}=k
\end{align*}
\end{proof}
With reference to the principal axis transformation, the eigenvectors
of the covariance matrix are then termed \textbf{principal components}
and affine linear subspaces of the embedding space as \textbf{linear
principal manifolds}.
\begin{defn*}[Linear Principal Manifold]
\label{def:Principal Component}\emph{For $d\in\mathbb{N}$ let $\boldsymbol{X}\colon\Omega\to\mathbb{R}^{d}$
be a random vector. Then a vector $\boldsymbol{c}\in\mathbb{R}^{d}$
with $\boldsymbol{c}\ne0$ is a principal component for $\boldsymbol{X}$,
iff there exists an $\lambda\in\mathrm{\mathbb{R}}$, such that :
\begin{equation}
\mathrm{Cov}(\boldsymbol{X})\cdot\boldsymbol{c}=\lambda\boldsymbol{c}\label{eq:Principal Component}
\end{equation}
Furthermore let $L\subseteq\mathbb{R}^{d}$ be an affine linear subspace
of $\mathbb{R}^{d}$ with dimension $k\leq d$. Then $L$ is a linear
}$k$-\emph{principal manifold for $\boldsymbol{X}$, if there exists
a set set $\boldsymbol{c}_{1},\,\ldots\,,\,\boldsymbol{c}_{k}$ of
linear independent principal components for $\boldsymbol{X}$, such
that:
\[
L=\mathbb{E}(\boldsymbol{X})+\bigoplus_{i=1}^{k}\mathbb{R}\boldsymbol{c}_{i}
\]
Then $L$ is termed maximal, iff the sum of the Eigenvalues $\lambda_{1},\,\ldots\,,\,\lambda_{k}$,
that correspond to the principal components $\boldsymbol{c}_{1},\,\ldots\,,\,\boldsymbol{c}_{k}$
is maximal. }
\end{defn*}
\begin{prop}
\label{prop:Principal Component Analysis}For $d\in\mathbb{N}$ let
$\boldsymbol{X}\colon\Omega\to\mathbb{R}^{d}$ be an elliptically
distributed random vector and $L\subseteq\mathbb{R}^{d}$ an affine
linear subspace. Then the following statements are equivalent: 

\emph{(i)} $L$ minimizes the Euclidean distance to $\boldsymbol{X}$

\emph{(ii)} $L$ is a maximal linear principal manifold for $\boldsymbol{X}$
\end{prop}

\begin{proof}
``$\Longrightarrow$'' Let $\pi_{L}\colon\mathbb{R}^{d}\hookrightarrow L$
denote the orthogonal projection of $\mathbb{R}^{d}$ onto $L$. Then
according to proposition \ref{prop:Affine fitting of Elliptic distributions}
$L$ minimizes the averaged Euclidean distance to $\boldsymbol{X}$,
if and only if (i) $\mathbb{E}(\boldsymbol{X})\in L$ and (ii) $L$
maximizes the variance $\mathrm{Var}(\pi_{L}(\boldsymbol{X}))$. In
particular (i) is satisfied, if and only if an orthonormal basis $\boldsymbol{u}_{1},\,\ldots\,,\,\boldsymbol{u}{}_{k}\in\mathbb{R}^{d}$
can be chosen, such that:
\[
L=\mathbb{E}(\boldsymbol{X})+\bigoplus_{i=1}^{k}\mathbb{R}\boldsymbol{u}_{i}
\]
Then according to Lemma \ref{lem:Variance of Projection 2} there
exist numbers $a_{1},\,\ldots\,,\,a_{d}\in[0,\,1]$ with $\sum_{i=1}^{d}a_{i}=k$,
such that:
\[
\mathrm{Var}(\pi_{L}(\boldsymbol{X}))=\sum_{i=1}^{d}\lambda_{i}a_{i}
\]
Thereupon (ii) is satisfied, if and only if the numbers $a_{i}$ maximize
this sum. Since the covariance matrix $\mathrm{Cov}(\boldsymbol{X})$
is positive semi-definite, the eigenvalues $\lambda_{i}$ are not
negative such that the sum is maximized for:

\begin{align*}
a_{i} & =\begin{cases}
1 & \text{for }i\in\{1,\,\ldots\,k\}\\
0 & \text{else}
\end{cases}
\end{align*}
Such that:

\begin{align*}
\sum_{j=1}^{k}\boldsymbol{u}_{j}^{T}\mathrm{Cov}(\boldsymbol{X})\boldsymbol{u}_{j} & =\sum_{i=1}^{d}\lambda_{i}a_{i}\\
 & =\sum_{i=1}^{k}\lambda_{i}\\
 & =\sum_{j=1}^{k}\boldsymbol{c}_{j}^{T}\mathrm{Cov}(\boldsymbol{X})\boldsymbol{c}_{j}
\end{align*}
Accordingly the choice $\boldsymbol{u}_{j}=\boldsymbol{c}_{j}$ for
$i\in\{1,\,\ldots\,,\,k\}$ maximizes $\mathrm{Var}(\pi_{L}(\boldsymbol{X}))$
and $L$ has a representation, given by:
\[
L=\mathbb{E}(\boldsymbol{X})+\bigoplus_{i=1}^{k}\mathbb{R}\boldsymbol{c}_{i}
\]
''$\Longleftarrow$'' Let $L$ have a representation as given by
(ii), then (1) $\mathbb{E}(\boldsymbol{X})\in L$ and (2) the variance
$\mathrm{Var}(\pi_{L}(\boldsymbol{X}))$ is maximized. According to
Proposition \ref{prop:Affine fitting of Elliptic distributions} it
follows, that $L$ minimizes the Euclidean distance to $\boldsymbol{X}$.
\end{proof}
\begin{defn*}[$L$-Correlation]
\label{def:L-Correlation}\emph{For $d\in\mathbb{N}$ let $\boldsymbol{X}\colon\Omega\to\mathbb{R}^{d}$
be a random vector and $L$ a maximal linear principal manifold for
$\boldsymbol{X}$. Then for any $i,\,j\in\{1,\,\ldots\,,\,d\}$ let
the $L$-Correlation between $X_{i}$ and $X_{j}$ be defined by:
\begin{equation}
\rho_{X_{i},X_{j}\mid L}^{2}\coloneqq R_{i}R_{j}\label{eq:L-Correlation 1}
\end{equation}
where with the orthogonal projection }$\pi_{L}\colon\mathbb{R}^{d}\to L$\emph{
for any $i\in\{1,\,\ldots\,,\,d\}$ the }\textbf{\emph{reliability}}\emph{
of $X_{i}$ with respect to $L$ is given by:
\begin{equation}
R_{i}\coloneqq1-\frac{\mathrm{Var}_{i}\left(\boldsymbol{X}-\pi_{L}(\boldsymbol{X})\right)}{\mathrm{Var}_{i}(\boldsymbol{X})}\label{eq:L-Correlation 2}
\end{equation}
}
\end{defn*}
\begin{prop}
\label{prop:L-Correlation}For an elliptically distributed random
vector the $L$-Correlation generalizes the Pearson Correlation to
maximal linear principal manifolds.
\end{prop}

\begin{proof}
Let $\boldsymbol{X}\colon\Omega\to\mathbb{R}^{2}$ be an elliptically
distributed random vector and $L$ a maximal linear $1$-principal
manifold for $\boldsymbol{X}$. Then for $i\in\{1,\,2\}$ the random
error of the variable $X_{i}$ has a variance:
\[
\eta_{X_{i}}^{2}=\mathrm{Var}_{X_{i}}\left(\boldsymbol{X}-\pi_{L}(\boldsymbol{X})\right)
\]
Such that by the definition of the reliability it follows, that:
\[
R_{i}\stackrel{{\scriptstyle \text{\ref{eq:L-Correlation 2}}}}{=}1-\frac{\eta_{X_{i}}^{2}}{\sigma_{X_{i}}^{2}}
\]
Consequently:
\[
\rho_{X_{i},X_{j}\mid L}^{2}=\left(1-\frac{\eta_{X_{i}}^{2}}{\sigma_{X_{i}}^{2}}\right)\left(1-\frac{\eta_{X_{j}}^{2}}{\sigma_{X_{j}}^{2}}\right)
\]
With $n\in\mathbb{N}$ i.i.d. realizations $\boldsymbol{x}\in\mathbb{R}^{n\times2}$
of $\boldsymbol{X}$ an empirical $L$-Correlation $\rho_{x_{i},x_{j}\mid L}^{2}$
is then given by replacing the variances by the sample variances.
Then by proposition \ref{prop:Correlation Reliability} it follows,
that:
\[
\rho_{x_{i},x_{j}}^{2}\overset{{\scriptstyle P}}{\to}\rho_{x_{i},x_{j}\mid L}^{2},\text{ for }n\to\infty
\]
\end{proof}

\section{The Riemann-Pearson Correlation}

Linear principal manifolds allow the projection of a random vector
$\boldsymbol{X}\colon\Omega\to\mathbb{R}^{d}$ onto a linear subspace
$L\subseteq\mathbb{R}^{d}$, which maximally preserves the linear
dependency structure of $\boldsymbol{X}$ in terms of its covariances.
Thereby for the orthogonal projection $\pi_{L}\colon\mathbb{R}^{d}\hookrightarrow L$,
the variance on $L$, given by $\mathrm{Var}(\pi_{L}(\boldsymbol{X}))$,
is referred as the \textbf{explained variance} and the orthogonal
deviation $\mathrm{Var}(\boldsymbol{X}-\pi_{L}(\boldsymbol{X}))$
as the \textbf{unexplained variance}. Thereupon by the assumption,
that $\boldsymbol{X}$ is elliptically distributed, it can be concluded,
that linear independence coincides with statistically independence,
that that $\pi_{L}(\boldsymbol{X})$ and $\boldsymbol{X}-\pi_{L}(\boldsymbol{X})$
are statistically independent and therefore allow the following decomposition:
\[
\underset{{\scriptstyle \text{total variance}}}{\underbrace{\mathrm{Var}(\boldsymbol{X})}}=\underset{{\scriptstyle \text{explained variance}}}{\underbrace{\mathrm{Var}(\pi_{L}(\boldsymbol{X}))}}+\underset{{\scriptstyle \text{unexplained variance}}}{\underbrace{\mathrm{Var}(\boldsymbol{X}-\pi_{L}(\boldsymbol{X}))}}
\]
This decomposition, as shown by theorem \ref{prop:L-Correlation},
is of fundamental importance for the correlation over linear Principal
Manifolds, since it determines the reliabilities of the respective
random variable $X$ by the ratio:

\[
R=1-\frac{{\scriptstyle \text{explained variance}}}{{\scriptstyle \text{total variance}}}
\]
\begin{figure}[h]
\begin{centering}
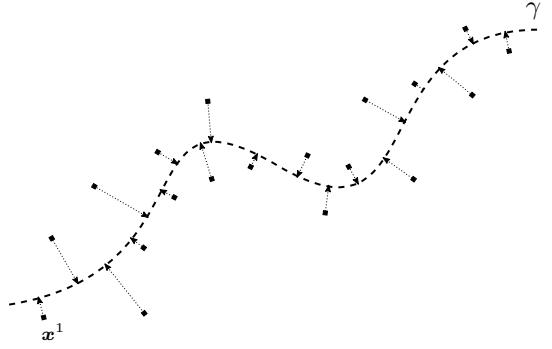
\par\end{centering}
\caption{\label{fig:Principal Curves}Principal Curve for a 2-dimensional realization}
\end{figure}
On this point of the discussion it's just a small step
to generalize the principal components, by a smooth curves $\gamma\colon[a,\,b]\to\mathbb{R}^{d}$
(figure \ref{fig:Principal Curves}). This is particular appropriate,
if the assumption of an elliptically distribution can only hardly
be justified, like for \textbf{observed dynamical systems}. Thereby
the evolution function generates a smooth submanifold $\mathcal{M}\subseteq\mathbb{R}^{d}$
within the observation space $\mathbb{R}^{d}$, and an ``error free''
observation can be identified by a random vector $\boldsymbol{X}^{\star}$,
with outcomes on $\mathcal{M}$. Additionally, however, the observation
function may be regarded to be subjected to a measurement error $\boldsymbol{\varepsilon}$.
By the assumption, that $\boldsymbol{\varepsilon}$ has an elliptical
distribution, then the distribution of the observable random vector
$\boldsymbol{X}$ is represented by a elliptical $\mathcal{M}-$distribution.
\begin{defn*}[$\mathcal{M}$-Distribution]
\label{def:M-Distribution}\emph{For $d\in\mathbb{N}$ let $\boldsymbol{X}^{\star}\colon\Omega\to\mathbb{R}^{d}$
be a random vector and $\mathcal{M}\subseteq\mathbb{R}^{d}$ a smooth
$k-$submanifold of $\mathbb{R}^{d}$ with $k\leq d$. Then $\boldsymbol{X}^{\star}$
is $\mathcal{M}-$distributed, iff for the probability density $P$,
which is induced by $\boldsymbol{X}^{\star}$, it holds, that: 
\begin{equation}
P(\boldsymbol{X}^{\star}=\boldsymbol{x})>0\Leftrightarrow\boldsymbol{x}\in\mathcal{M}\label{eq:M-Distribution}
\end{equation}
Thereupon a random vector $\boldsymbol{X}\colon\Omega\to\mathbb{R}^{d}$
is elliptically $\mathcal{M}$-distributed, iff there exists an $\mathcal{M}$-distributed
random vector $\boldsymbol{X}^{\star}\colon\Omega\to\mathbb{R}^{d}$
and an elliptically distributed random error $\boldsymbol{\varepsilon}\colon\Omega\to\mathbb{R}^{d}$,
such that:
\begin{equation}
\boldsymbol{X}\sim\boldsymbol{X}^{\star}+\boldsymbol{\varepsilon}\label{eq:Elliptical M-Distribution}
\end{equation}
}
\end{defn*}
\begin{figure}[h]
\begin{centering}
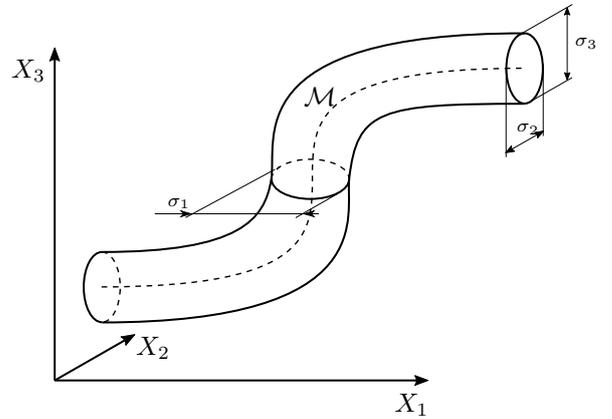
\par\end{centering}
\caption{\label{fig:Smooth Tubular Distributions}Elliptical $\mathcal{M}$-distribution
in $3$ dimensions}
\end{figure}
The assumption, that the observed random vector $\boldsymbol{X}$,
is elliptically $\mathcal{M}$-distributed, is very general, but allows
an estimation of $\mathcal{M}$ by minimizing the averaged Euclidean
distance to $\boldsymbol{X}$. Thereby the tangent spaces $T_{x}\mathcal{M}$
have a basis, given by $k$ principal components of local infinitesimal
covariances, such that the remaining $d-k$ principal components describe
the normal space $N_{x}\mathcal{M}$, which is orthogonal to the tangent
space $T_{x}\mathcal{M}$. Since $T_{x}\mathcal{M}$ and $N_{x}\mathcal{M}$
are equipped with an induced Riemannian metric, which is simply given
by the standard scalar product, there exists a minimal orthogonal
projection $\pi_{\mathcal{M}}\colon\mathbb{R}^{d}\hookrightarrow\mathcal{M}$,
that maps any realization $\boldsymbol{x}$ of $\boldsymbol{X}$ to
a closest point on $\mathcal{M}$. Then proposition \ref{prop:Affine fitting of Elliptic distributions}
motivates properties for $\mathcal{M}$ to minimize the averaged Euclidean
distance to realizations of $\boldsymbol{X}$. This provides the definition
of smooth $k$-principal manifolds (\nocite{Hastie1989}Hastie et
al. 1989, p513).
\begin{defn*}[Principal Manifold]
\label{def:Principal Manifold}\emph{For $d\in\mathbb{N}$ let $\boldsymbol{X}\colon\Omega\to\mathbb{R}^{d}$
be a random vector, $\mathcal{M}\subseteq\mathbb{R}^{d}$ a (smooth)
$k$-submanifold of $\mathbb{R}^{d}$ with $k\leq d$ and $\pi_{\mathcal{M}}\colon\mathbb{R}^{d}\hookrightarrow\mathcal{M}$
a minimal orthogonal projection onto $\mathcal{M}$. Then $\mathcal{M}$
is a (smooth) $k$-principal manifold for $\boldsymbol{X}$, iff $\forall\boldsymbol{x}\in\mathcal{M}$
it holds, that: 
\begin{equation}
\mathbb{E}(\boldsymbol{X}\in\pi_{\mathcal{M}}^{-1}(\boldsymbol{x}))=\boldsymbol{x}\label{eq:Principal Manifold}
\end{equation}
Furthermore $\mathcal{M}$ is termed maximal, iff $\mathcal{M}$ maximizes
the explained variance $\mathrm{Var}(\pi_{\mathcal{M}}(\boldsymbol{X}))$.}
\end{defn*}
By extending the local properties of the tangent spaces to the underlying
manifold, by propositions \ref{prop:Affine fitting of Elliptic distributions}
and \ref{prop:Principal Component Analysis} it can be concluded,
that maximal principal manifolds minimize the Euclidean distance to
$\boldsymbol{X}$. Intuitively this can be understood as follows:
The principal manifold property assures, that:

\[
\mathrm{Var}(\boldsymbol{X})=\mathrm{Var}(\pi_{\mathcal{M}}(\boldsymbol{X}))+\mathrm{Var}(\boldsymbol{X}-\pi_{\mathcal{M}}(\boldsymbol{X}))
\]
Consequently the choice of $\mathcal{M}$ maximizes $\mathrm{Var}(\pi_{\mathcal{M}}(\boldsymbol{X}))$
if and only if it minimizes $\mathrm{Var}(\boldsymbol{X}-\pi_{\mathcal{M}}(\boldsymbol{X}))$,
which equals the variance of the error and therefore the Euclidean
distance. At closer inspection, however, it turns out, that in difference
to linear principal manifolds, the maximization problem is ill-defined
for arbitrary smooth principal manifolds, since for any finite number
of realizations trivial solutions can be found by smooth principal
manifolds, that interpolate the realizations and therefore provide
a perfect explanation. In order to close this gap, further structural
structural assumptions have to be incorporated, ether by a parametric
family $\{\boldsymbol{f}_{\theta}\}_{\theta\in\Theta}$ that restricts
the possible solutions - or by a regularization, as given in the elastic
map algorithm that penalizes long distances and strong curvature (\nocite{Gorban2008}Gorban
et al. 2008). Due to the complexity of this topic, however, it is
left to the second chapter, where Energy based models are used to
overcome this deficiency. In the following the generalization of the
correlation to smooth principal manifolds for convenience is defined
with respect to a principal manifold $\mathcal{M}$, which is maximal
``with respect to appropriate restrictions''.
\begin{defn*}[Riemann-Pearson Correlation]
\label{def:M-Correlation}\emph{For $d\in\mathbb{N}$ let $\boldsymbol{X}\colon\Omega\to\mathbb{R}^{d}$
be a random vector, $\mathcal{M}$ a smooth principal manifold for
$\boldsymbol{X}$, which is maximal ``with respect to appropriate
restrictions'' and $\pi_{\mathcal{M}}\colon\mathbb{R}^{d}\to\mathcal{M}$
a minimal orthogonal projection. Then for any $i,\,j\in\{1,\,\ldots\,,\,d\}$
the Riemann-Pearson Correlation between $X_{i}$ and $X_{j}$ is given
by:
\begin{equation}
\rho_{X_{i},X_{j}\mid\mathcal{M}}^{2}\coloneqq R_{i}R_{j}\int_{\mathcal{M}}S_{i,j}(\boldsymbol{x})S_{j,i}(\boldsymbol{x})\,\mathrm{d}P_{\mathcal{M}}\label{eq:M-Correlation}
\end{equation}
where for $i\in\{1,\,\ldots\,,\,d\}$ the }\textbf{\emph{reliability}}\emph{
of $X_{i}$ with respect to $\mathcal{M}$ is given by:
\begin{equation}
R_{i}\coloneqq1-\frac{\mathrm{Var}_{i}\left(\boldsymbol{X}-\pi_{\mathcal{M}}(\boldsymbol{X})\right)}{\mathrm{Var}_{i}(\boldsymbol{X})}\label{eq:M-Correlation 2}
\end{equation}
and for $i,\,j\in\{1,\,\ldots\,,\,d\}$ the }\textbf{\emph{local sensitivity}}\emph{
of $X_{i}$ with respect to $X_{j}$ by:
\begin{equation}
S_{i,j}(\boldsymbol{x})\coloneqq\left.\frac{\partial}{\partial x_{j}}(\boldsymbol{x}-\pi_{\mathcal{M}}(\boldsymbol{x}))\right|_{i}\label{eq:M-Correlation 3}
\end{equation}
}
\end{defn*}
\begin{prop}
\label{Prop:M-Correlation}For an elliptical $\mathcal{M}-$distributed
random vector $\boldsymbol{X}$ the Riemann-Pearson Correlation generalizes
the $L$-Correlation to smooth principal manifolds.
\end{prop}

\begin{proof}
Let $L$ be a maximal linear principal manifold for $\boldsymbol{X}\colon\Omega\to\mathbb{R}^{d}$,
and $\boldsymbol{x}$ a realization of $\boldsymbol{X}$ then there
exists an $\beta\in\mathbb{R}$ with:
\[
S_{i,j}(\boldsymbol{x})=\left.\frac{\partial}{\partial x_{j}}(\boldsymbol{x}-\pi_{L}(\boldsymbol{x}))\right|_{i}\equiv\beta
\]
Furthermore for $c\ne0$:
\[
S_{j,i}(\boldsymbol{x})=\left.\frac{\partial}{\partial x_{i}}(\boldsymbol{x}-\pi_{L}(\boldsymbol{x}))\right|_{j}\equiv\frac{1}{\beta}
\]
Such that $S_{i,j}(\boldsymbol{x})S_{j,i}(\boldsymbol{x})=1$. Consequently
for $\mathcal{M}=L$ it follows, that:
\begin{align*}
\rho_{X_{i},X_{j}\mid\mathcal{M}}^{2} & =R_{i}R_{j}\int_{\mathcal{M}}S_{i,j}(\boldsymbol{x})S_{j,i}(\boldsymbol{x})\,\mathrm{d}P_{\mathcal{M}}\\
 & =R_{i}R_{j}\int_{\mathcal{M}}\,\mathrm{d}P_{\mathcal{M}}\\
 & =R_{i}R_{j}=\rho_{X_{i},X_{j}\mid L}^{2}
\end{align*}
\end{proof}
\bibliographystyle{unsrt}
\bibliography{articles}

\end{document}